\long\def\symbolfootnote[#1]#2{\begingroup\def\thefootnote{\fnsymbol{footnote}}\footnote[#1]{#2}\endgroup}

\documentclass{amsart}
\usepackage{hyperref}
\usepackage{indentfirst}
\usepackage{amssymb}
\usepackage{amsmath}
\usepackage{amsfonts}

\newtheorem{theorem}{Theorem}[section]
\newtheorem{corollary}[theorem]{Corollary}

\theoremstyle{remark}
\newtheorem{remark}[theorem]{Remark}
\theoremstyle{definition}

\theoremstyle{proposition}
\newtheorem{proposition}[theorem]{Proposition}
\numberwithin{equation}{section}

\begin{document}
\author{Linfeng Zhou}
\title[The Finsler surface with K=0 and J=0]{The Finsler surface with K=0 and J=0 }%\uppercase\expandafter{\romannumeral 1}.}
\date{}
\maketitle

\begin{abstract} In this short note, we verify Bryant's claim:  there does exist the singular Landsberg Finsler surface with a vanishing flag curvature which is not Berwaldian. 
\\

\end{abstract}

\section{Introduction}
In Finsler geometry, there is a famous unicorn open problem: does every regular Landsberg metric need to be Berwaldian?  If considering this problem in the more general situation, it is not true. Actually,  G.S. Asanov \cite{As} and Z. Shen \cite{Sh} have constructed the singular non-Berwaldian Landsberg Finsler metrics belonging to $(\alpha,\beta)$ metrics. Even earlier, R. Bryant claimed there do exist the singular Landsberg Finsler surfaces which are not Berwaldian. Furthermore, he announced that among them there is the surface depending on a function of two variables with a vanishing flag curvature \cite{Ba}. 

The main purpose of this paper is to prove the Bryant's argument by using a different approach. More precisely, we search the required Finsler surface among the spherically symmetric metrics defined on a domain in $R^2$. 

The spherically symmetric metrics belong to the generalized $(\alpha, \beta)$ metrics and  since they have a nice rotational symmetry, most geometric quantities become simple and neat. In the paper \cite{MZ},  the Berwald curvature, the Landsberg curvatrue, the Riemann curvature and the Ricci curvature of this type are discussed and some interesting results are obtained.  

In the section 2 of this paper, the mean Berwald curvature and the mean Landsberg curvature are also concerned. It is found that in two dimensional case, the mean Landsberg curvature can be easier. From it, one can observe that the equation of the Landsberg surfaces is weaker than the equation of the higher dimensional Landsberg metrics. %Perhaps, it is helpful to handle the unicorn problem.

When focusing on the  two dimensional Finsler surfaces, one class of  the possible geodesic spray coefficients of the Landsberg metrics, which arouses from a classification theorem of Landsberg metric of higher dimension \cite{MZ}, is studied. The following result is achieved: if the spray coefficient is given by $G^i=uPy^i+u^2Qy^i$  where $P=f_1(r)s+f_2(r)\sqrt{r^2-s^2}$ and $Q=c_0(r)+c_2(r)s^2+c_1(r)s\sqrt{r^2-s^2}$, then it must be Landsbergian. Furthermore,  if $f_2(r)\neq -\frac{r^2c_1(r)}{3}$, then it is not Berwaldian. 

The only frustrating thing is this metric is singular when a tangent vector is along the direction of  the radius. However, just looking at the equation (\ref{2-lce}) of the Landsberg surfaces, it is still hopeful to find some regular counterexamples of the unicorn problem. 

In the final section, by applying the constant Ricci curvature equation, the required singular Lansberg Finsler surface with K=0  is constructed. It just depends on a one variable function and a constant. This surface is chosen from the Landsberg surfaces in section 2 and is complicated a little bit. For the readers' convenience, we list it here:
 \[F=u \exp(\int_0^s\frac{(c+1)s^2-(2r^2c_0-1)s\sqrt{r^2-s^2}-2r^2c}{(c+1)s^3-(c+1)r^2s+(2c_0r^2-1)(r^2-s^2)\sqrt{r^2-s^2}} ds)a(r)\]
 is defined on a domain $\Omega$ in $R^2$, where $r=\sqrt{x_1^2+x_2^2}$, $u=\sqrt{y_1^2+y_2^2}$, $s=x_1y_1+x_2y_2$ and $$a(r)=\exp(\int-\frac{4c_0r^2r^4-4r^2c_0+1+2c(2c_0r^2-1)(c-1)}{r(4c_0r^4-4r^2c_0+1)}dr),$$
$c_0$ is a smooth function of $r$ and $c\neq\frac{1}{3}$ is a constant.

\section{The mean Berwald curvature and the mean Landsberg curvature}

Now let us consider a spherically symmetric Finsler metric $F=u\phi(r,s)$ defined on a domain $\Omega\subseteq R^n$ where $u=|y|$, $r=|x|$ and $s=\frac{\langle x, y\rangle}{|y|}$. As we know, its metric tensor is given by 
\begin{eqnarray}\label{mt}g_{ij}&=&\phi(\phi-s\phi_s)\delta_{ij}+(\phi_s^2+\phi\phi_{ss})x^ix^j+[s^2\phi\phi_{ss}-s(\phi-s\phi_s)\phi_s]\frac{y^i}{u}\frac{y^j}{u}\nonumber\\
&&+[(\phi-s\phi_s)\phi_s-s\phi\phi_{ss}](x^i\frac{y^j}{u}+x^j\frac{y^i}{u}).
\end{eqnarray} and its spray coefficients $G^i$ can be expressed as
\[G^i=uPy^i+u^2Qx^i\]
where 
$$P=-\frac{1}{\phi}\big(s\phi+(r^2-s^2)\phi_s\big)Q+\frac{1}{2r\phi}(s\phi_r+r\phi_s)$$
and 
$$Q=\frac{1}{2r}\frac{-\phi_r+s\phi_{rs}+r\phi_{ss}}{\phi-s\phi_s+(r^2-s^2)\phi_{ss}}.$$

In the previous paper \cite{MZ}, its Berwald curvature is computed  as following
\begin{eqnarray*}
B^i_{\ jkl}&=&\frac{P_{ss}}{u}(\delta^i_{\ j} x^kx^l+\delta^i_{\ l}x^jx^k+\delta^i_{\ k}x^jx^l)+(\frac{P}{u}-\frac{s}{u}P_s)(\delta^i_{\ j}\delta_{kl}+\delta^i_{\ k}\delta_{jl}+\delta^i_{\ l}\delta_{jk})\\
&&-\frac{s}{u^2}P_{ss}\big(\delta^i_{\ j}(x^ky^l+x^ly^k)+\delta^i_{\ k}(x^jy^l+x^ly^j)+\delta^i_{\ l}(x^jy^k+x^ky^j)\big)\\
&&-\frac{s}{u^2}P_{ss}y^i(\delta_{jk}x^l+\delta_{jl}x^k+\delta_{kl}x^j)+(\frac{Q_s}{u}-\frac{s}{u}Q_{ss})x^i(\delta_{jk}x^l+\delta_{jl}x^k+\delta_{kl}x^j)\\
&&+(\frac{s^2}{u^3}P_{ss}+\frac{s}{u^3}P_s-\frac{P}{u^3})(\delta^i_{\ j}y^ky^l+\delta^i_{\ k}y^jy^l+\delta^i_{\ l}y^jy^k)\\
&&+(\frac{s^2}{u^3}P_{ss}+\frac{s}{u^3}P_s-\frac{P}{u^3})y^i(\delta_{jk}y^l+\delta_{jl}y^k+\delta_{kl}y^j)\\
&&+(\frac{3}{u^5}P-\frac{s^3}{u^5}P_{sss}-\frac{6s^2}{u^5}P_{ss}-\frac{3s}{u^5}P_s)y^iy^jy^ky^l\\
&&+(\frac{s^2}{u^4}P_{sss}+\frac{3s}{u^4}P_{ss})y^i(y^jy^kx^l+y^jy^lx^k+y^ky^lx^j)+\frac{P_{sss}}{u^2}y^ix^jx^kx^l\\
&&-(\frac{P_{ss}}{u^3}+\frac{s}{u^3}P_{sss})y^i(y^jx^kx^l+y^kx^jx^l+y^lx^jx^k)\\
&&+(\frac{s^2}{u^3}Q_{sss}+\frac{s}{u^3}Q_{ss}-\frac{Q_s}{u^3})x^i(x^jy^ky^l+x^ky^jy^l+x^ly^jy^k)\\
&&-\frac{s}{u^2}Q_{sss}x^i(x^jx^ly^k+x^jx^ky^l+x^kx^ly^j)+\frac{Q_{sss}}{u}x^ix^jx^kx^l\\
&&+(\frac{s^2}{u^2}Q_{ss}-\frac{s}{u^2}Q_s)x^i(\delta_{kl}y^j+\delta_{jl}y^k+\delta_{jk}y^l)\\
&&+(\frac{3s}{u^4}Q_s-\frac{3s^2}{u^4}Q_{ss}-\frac{s^3}{u^4}Q_{sss})x^iy^jy^ky^l
\end{eqnarray*}
and it is proved that all Berwald metrics of this type must be Riemannian when $n\geq 3$.  

The mean Berwald curvature is the trace of Berwald curvature and it is defined as
$E:=E_{ij}dx^i\otimes dx^j$ where 
$E_{ij}:=B^m_{\ ijm}$. Plugging the Berwald curvature into the definition, one can get the formula 
\begin{eqnarray*}
E_{ij}&=&\frac{\delta_{ij}}{u}\big((n+1)(P-sP_s)+(r^2-s^2)(Q_s-sQ_{ss})\big)\\
&&+\frac{y^iy^j}{u^3}\big((n+1)(s^2P_{ss}+sP_s-P)+r^2(s^2Q_{sss}+sQ_{ss}-Qs)\\
&&\qquad+3s^2Q_s-3s^3Q_{ss}-s^4Q_{sss}\big)\\
&&+\frac{x^ix^j}{u}\big((n+1)P_{ss}+2(Q_s-sQ_{ss})+(r^2-s^2)Q_{sss}\big)\\
&&-\frac{(x^iy^j+x^jy^i)}{u^2}\big( (n+1)P_{ss}+2(Q_s-sQ_{ss})+(r^2-s^2)Q_{sss}\big)s.
\end{eqnarray*}
A Finsler metric is called weakly Berwaldian if its mean Berwald curvature vanishes.

\begin{proposition} \label{mbp} A spherically symmetric Finsler metric $(\Omega, F)$ in $R^n$ is weakly Berwaldian if and only its spray coefficient satisfies the equation
\begin{equation}\label{mbe}
(n+1)(P-sP_s)+(r^2-s^2)(Q_s-sQ_{ss})=0.\end{equation}
\end{proposition}
\begin{proof} From the formula of the mean Berwald curvature, it is easy to see that the metric is weakly Berwaldian if and only if 
\begin{equation*}
\left\{ \begin{array}{l}
(n+1)(P-sP_s)+(r^2-s^2)(Q_s-sQ_{ss})=0\\
\\
(n+1)(s^2P_{ss}+sP_s-P)+r^2(s^2Q_{sss}+sQ_{ss}-Qs)+3s^2Q_s-3s^3Q_{ss}-s^4Q_{sss}=0\\
\\
(n+1)P_{ss}+2(Q_s-sQ_{ss})+(r^2-s^2)Q_{sss}=0.\\
\end{array}\right.
\end{equation*}
Notice that if differentiating the first equation will yield  the third equation of above system and adding the first equation and the third equation will give the second equation. Hence the equation system is equivalent to 
$$(n+1)(P-sP_s)+(r^2-s^2)(Q_s-sQ_{ss})=0.$$
\end{proof}
\begin{remark} From the equation (\ref{mbe}), it is not difficult for one to find those weak Berwald metrics which are not Berwaldian.  
\end{remark}

The Landsberg curvature of a spherically symmetric Finsler metric has been already be known as \cite{MZ}
\begin{eqnarray*}L_{jkl}&=&-\frac{\phi}{2}[L_1x^jx^kx^l+L_2(x^j\delta_{kl}+x^k\delta_{jl}+x^l\delta_{jk})+L_3\frac{y^j}{u}\frac{y^k}{u}\frac{y^l}{u}\\
&&+L_4(\frac{y^j}{u}\delta_{kl}+\frac{y^k}{u}\delta_{jl}+\frac{y^l}{u}\delta_{jk})+L_5(\frac{y^j}{u}x^kx^l+\frac{y^k}{u}x^jx^l+\frac{y^l}{u}x^jx^k)\\
&&+L_6(x^j\frac{y^k}{u}\frac{y^l}{u}+x^k\frac{y^j}{u}\frac{y^l}{u}+x^l\frac{y^j}{u}\frac{y^k}{u})]
\end{eqnarray*}
where 
\begin{eqnarray*}
L_1&=&3\phi_sP_{ss}+\phi P_{sss}+\big(s\phi+(r^2-s^2)\phi_s\big)Q_{sss},\\
L_2&=&-s\phi P_{ss}+\phi_s(P-sP_s)+(s\phi+(r^2-s^2)\phi_s)(Q_s-sQ_{ss}),\\
L_3&=&-s^3L_1+3sL_2,\\
L_4&=&-sL_2,\\
L_5&=&-sL_1,\\
L_6&=&s^2L_1-L_2.
\end{eqnarray*}
Furthermore, when $n\geq3$, we classify the Landsberg metrics of this type \cite{MZ}.

The mean Landsberg curvature is the trace of the Landsberg curvature and it is given by 
$J:=J_idx^i$ where $J_i:=L_{ijk}g^{jk}$ and $g^{jk}$ is the inverse of the metric tensor $g_{jk}$. Since $g_{ij}$ is shown in (\ref{mt}), one can calculate its inverse $g^{ij}$:
\[g^{ij}=\tilde{\rho}_0\delta^{ij}+\tilde{\rho}_1\frac{y^iy^j}{u^2}+\tilde{\rho}_2(\frac{x^iy^j}{u}+\frac{x^jy^i}{u})+\tilde{\rho}_3x^ix^j\]
where $\tilde{\rho}_0=\frac{1}{\phi(\phi-s\phi_s)}$, $\tilde{\rho}_1=\frac{(s\phi+(r^2-s^2)\phi_s)(\phi\phi_s-s\phi_s^2-s\phi\phi_{ss})}{\phi^3(\phi-s\phi_s)(\phi-s\phi_s+(r^2-s^2)\phi_{ss})}$, $\tilde{\rho}_2=-\frac{\phi\phi_s-s\phi_s^2-s\phi\phi_{ss}}{\phi^2(\phi-s\phi_s)(\phi-s\phi_s+(r^2-s^2)\phi_{ss})}$, $\tilde{\rho}_3=-\frac{\phi_{ss}}{\phi(\phi-s\phi_s)(\phi-s\phi_s+(r^2-s^2)\phi_{ss})}$. Then the mean Landsberg curvature can be written as
\[J_i=x^i J_1+\frac{y^i}{u}J_2\]
where 
\begin{eqnarray*}
J_1&=&-\frac{\phi}{2}[(r^2-s^2)\big(\tilde{\rho}_0+(r^2-s^2)\tilde{\rho}_3\big)L_1+\big((n+1)\tilde{\rho}_0+3(r^2-s^2)\tilde{\rho}_3\big)L_2],\\
J_2&=&-sJ_1,\\
L_1&=&3\phi_sP_{ss}+\phi P_{sss}+\big(s\phi+(r^2-s^2)\phi_s\big)Q_{sss},\\
L_2&=&-s\phi P_{ss}+\phi_s(P-sP_s)+(s\phi+(r^2-s^2)\phi_s)(Q_s-sQ_{ss}).
\end{eqnarray*}
A Finsler metric is called weakly Landsbergian if its mean Landsberg curvature vanishes. From the formula of the mean Landsberg curvature, one can know that:
\begin{proposition}\label{mlp}A spherically symmetric Finsler metric $(\Omega, F)$ in $R^n$ is the weakly Landsberg metric if and only its spray coefficient satisfies the equation
\begin{equation}\label{mle}
(r^2-s^2)\big(\tilde{\rho}_0+(r^2-s^2)\tilde{\rho}_3\big)L_1+\big((n+1)\tilde{\rho}_0+3(r^2-s^2)\tilde{\rho}_3\big)L_2=0\end{equation}
where 
\begin{eqnarray*}
\tilde{\rho}_0&=&\frac{1}{\phi(\phi-s\phi_s)},\\
\tilde{\rho}_3&=&-\frac{\phi_{ss}}{\phi(\phi-s\phi_s)(\phi-s\phi_s+(r^2-s^2)\phi_{ss})},\\
L_1&=&3\phi_sP_{ss}+\phi P_{sss}+\big(s\phi+(r^2-s^2)\phi_s\big)Q_{sss},\\
L_2&=&-s\phi P_{ss}+\phi_s(P-sP_s)+(s\phi+(r^2-s^2)\phi_s)(Q_s-sQ_{ss}).\\
\end{eqnarray*}
\end{proposition}

\begin{remark} The equation of weak Landsberg metric is weaker than the equation of Landsberg metric: $L_1=0$ and $L_2=0$ and there should exist many weak Landsberg metrics which are not Landsbergian when the dimension $n\geq 3$. 
\end{remark}

Actually, in two dimensional case, the equation can be quite simple and we have the following corollary. 

\begin{corollary}\label{2-lcc}A spherically symmetric Finsler surface $(\Omega, F)$ in $R^2$ is the Landsberg metric if and only its spray coefficient satisfies the equation
\begin{equation}\label{2-lce} (r^2-s^2)L_1+3L_2=0\end{equation}
where 
\begin{eqnarray*}
L_1&:=&3\phi_sP_{ss}+\phi P_{sss}+\big(s\phi+(r^2-s^2)\phi_s\big)Q_{sss},\\
L_2&:=&-s\phi P_{ss}+\phi_s(P-sP_s)+(s\phi+(r^2-s^2)\phi_s)(Q_s-sQ_{ss}).\\
\end{eqnarray*}
\end{corollary}

\begin{proof} The proposition \ref{mlp} tells us the Finsler surface $F$ is Landsbergian if and only if 
\[(\tilde{\rho}_0+(r^2-s^2)\tilde{\rho}_3)((r^2-s^2)L_1+3L_2)=0.\]
Meanwhile, one can see that 
$$\tilde{\rho}_0+(r^2-s^2)\tilde{\rho}_3=\frac{1}{\phi-s\phi_s+(r^2-s^2)\phi_{ss}}\neq 0.$$ 
Thus $(r^2-s^2)L_1+3L_2=0.$

\end{proof}

\begin{remark} 
(1) As we know,  2-dimensional Finsler surface is Berwaldian if and only if it is Landsbergain and weakly Berwaldian.  Actually the equation (\ref{2-lce}) is weaker than the equation (\ref{mbe}) and one can see this via the following theorem \ref{2-le} although the metrics are singular. Hence it seems possible to find the regular Landsberg surface of this type which is not Berwaldian. 

(2) It is a challenge to write out all the solutions of the equation (\ref{2-lce}) since it involves two functions of two variables.  

\end{remark}

\begin{theorem}\label{2-le} Let $(\Omega, F)$ be a spherically symmetric Finsler surface in $R^2$. If assume its spray coefficient is given by $$G^i=uPy^i+u^2Qx^i$$ where $$P=f_1(r)s+f_2(r)\sqrt{r^2-s^2}$$ and $$Q=c_0(r)+c_2(r)s^2+c_1(r)s\sqrt{r^2-s^2},$$ then it is Landsbergian. If $f_2(r)\neq -\frac{r^2c_1(r)}{3}$, then it is not Berwaldian.
\end{theorem}

\begin{proof}
According to the corollary \ref{2-lcc}, to prove $F$ is Landsbergian, we only need to verify that $P$ and $Q$ satisfy the equation
\begin{equation}\label{eq1}(r^2-s^2)L_1+3L_2=0\end{equation}
where 
\begin{eqnarray*}
L_1&:=&3\phi_sP_{ss}+\phi P_{sss}+\big(s\phi+(r^2-s^2)\phi_s\big)Q_{sss},\\
L_2&:=&-s\phi P_{ss}+\phi_s(P-sP_s)+(s\phi+(r^2-s^2)\phi_s)(Q_s-sQ_{ss}).\\
\end{eqnarray*}
However, this equation involves the metric function $F=u\phi(r,s)$. Thus it is necessary to express $\frac{\phi_s}{\phi}$ in terms of $P$ and $Q$. In order to do that, let us introduce 
 $U$ and $W$ so that
\[U:=\frac{s\phi+(r^2-s^2)\phi_s}{\phi},\qquad W:=\frac{s\phi_r+r\phi_s}{\phi}.\]
From the definition, it is easy to solve
\begin{equation*}
\phi_s=\frac{U-s}{r^2-s^2}\phi,\qquad\phi_r=\frac{1}{s}(W-\frac{r(U-s)}{r^2-s^2})\phi.
\end{equation*}
Plugging $\phi_s$ and $\phi_r$ into the formula of $P=-\frac{1}{\phi}\big(s\phi+(r^2-s^2)\phi_s\big)Q+\frac{1}{2r\phi}(s\phi_r+r\phi_s)$ and $Q=\frac{1}{2r}\frac{-\phi_r+s\phi_{rs}+r\phi_{ss}}{\phi-s\phi_s+(r^2-s^2)\phi_{ss}}$, one can get the expression
\begin{equation*}
\left\{ \begin{array}{l}
 P=-QU+\frac{W}{2r}\\
 \\
Q=\frac{1}{2rs}\frac{2rU-2rs-2r^2W+s(r^2-s^2)W_s+s^2W+sUW}{U^2-sU+(r^2-s^2)U_s}.   
 \end{array} \right.
          \end{equation*}
Now $P=f_1(r)s+f_2(r)\sqrt{r^2-s^2}$ and $Q=c_0(r)+c_2(r)s^2+c_1(r)s\sqrt{r^2-s^2}$. From above equalities, one can solve that 
$$U=\frac{s\sqrt{r^2-s^2}(r^2f_1+1)+2r^2f_2(r^2-s^2)}{\sqrt{r^2-s^2}(f_1s^2-2c_0(r^2-s^2)+1)+c_1r^2s^3+f_2(r^2-s^2)s-c_1r^4s}.$$
Therefore
$$\frac{\phi_s}{\phi}=\frac{s\sqrt{r^2-s^2}(2c_0+f_1)+2f_2r^2+(c_1r^2-f_2)s^2}{\sqrt{r^2-s^2}(f_1s^2-2c_0(r^2-s^2)+1)+c_1r^2s^3+f_2(r^2-s^2)s-c_1r^4s}.$$
Plugging $P$, $Q$ and $\frac{\phi_s}{\phi}$ into (\ref{eq1}), one will find the equation holds.  

On the other hand, we know that for two dimensional Finsler surface, it is Berwaldian if and only if it is Langangian and weakly Berwalidan \cite{Sh1}. When $P=f_1(r)s+f_2(r)\sqrt{r^2-s^2}$ and $Q=c_0(r)+c_2(r)s^2+c_1(r)s\sqrt{r^2-s^2}$, one can compute 
$$3(P-sP_s)+(r^2-s^2)(Q_s-sQ_{ss})=\frac{(c_1r^2+3f_2)r^2}{\sqrt{r^2-s^2}}.$$
Obviously if $f_2(r)\neq -\frac{r^2c_1(r)}{3}$, it is not weakly Berwaldian. So it is not Berwaldian.
\end{proof}

\begin{remark} Conversely, by applying the same method, one can prove the result: if $P=f_1(r)s+f_2(r)\sqrt{r^2-s^2+g(r)}$ and $Q=c_0(r)+c_2(r)s^2+c_1(r)s\sqrt{r^2-s^2+g(r)}$ is the spray coefficient of a spherically symmetric Finsler surface and if it is Landsbergian, then $g(r)=0$ or $f_2(r)=c_1(r)=0$. This is the reason why we pose the condition on $P$ and $Q$ in above theorem. 
\end{remark}

\section{The spherically symmetric Finsler surface with J=0 and K=0}

For a spherically symmetric Finsler metric $F=u\phi(r,s)$ on a domain $\Omega$ in $R^n$, its Ricci curvature  can be expressed in terms of its spray coefficient $P$ and $Q$ as \cite{MZ}
\[\bold{Ric}=(n-1)R_1+(r^2-s^2)R_3 \]
where 
\[R_1=2Q-\frac{s}{r}P_r-P_s+2(r^2-s^2)P_sQ+P^2+2sPQ\]
and \[R_3=\frac{2}{r}Q_r-Q_{ss}-\frac{s}{r}Q_{rs}+2(r^2-s^2)QQ_{ss}+4Q^2-(r^2-s^2)Q_s^2-2sQQ_s.\]

When in two dimensional case, the flag curvature of a Finsler surface only depends on the point and the direction of a flag's pole. It automatically has of scalar curvature. This observation implies the class of the Finsler metrics with a constant flag curvature equals to the class  of the metrics with constant Ricci curvature in two dimensional case.  If considering a spherically symmetric Finsler surface with a vanishing flag curvature,  we have the following theorem.
\begin{theorem}\label{cfct}  Let  $(\Omega, F)$  be a spherically symmetric Finsler surface in $R^2$. It has a vanishing flag curvature if and only if its spray coefficient satisfies
\begin{equation}\label{eq2}
R_1+(r^2-s^2)R_3=0\end{equation}
where 
\begin{eqnarray*}
R_1&=&2Q-\frac{s}{r}P_r-P_s+2(r^2-s^2)P_sQ+P^2+2sPQ,\\
R_3&=&\frac{2}{r}Q_r-Q_{ss}-\frac{s}{r}Q_{rs}+2(r^2-s^2)QQ_{ss}+4Q^2-(r^2-s^2)Q_s^2-2sQQ_s.
\end{eqnarray*}
\begin{remark} The equation in above theorem is weaker than the characterized equation of constant flag curvature in higher dimension \cite{MZ}.
\end{remark}

\end{theorem}

Now it is ready for us to construct the Finsler surface with J=0 and K=0. However, here we would prefer specify its spray coefficient first instead of the metric function since every quantity can be formulated by the spray coefficient.  

According to the theorem \ref{2-le},  if the spray coefficient of a spherically symmetric Finsler surface is given by 
$$P=f_1(r)s+f_2(r)\sqrt{r^2-s^2}$$ and $$ Q=c_0(r)+c_2(r)s^2+c_1(r)s\sqrt{r^2-s^2},$$
then it is Landsbergian. Plugging $P$ and $Q$ into the equation (\ref{eq2}) in theorem \ref{cfct} will obtain
\[A_3(r)s^3+A_2(r)s^2\sqrt{r^2-s^2}+A_1(r)s+A_0(r)\sqrt{r^2-s^2}=0\]
where 
\begin{eqnarray*}
A_3(r)&=&-(6rc_1+r^2c_1'+2r^3c_1f_1+2rf_1f_2-f_2'),\\
A_2(r)&=&-(4rc_0^2+2c_0'+4r^3c_0c_2-r^5c_1^2-4rc_2-2r^3c_2f_1-rf_1^2+f_1'+rf_2^2),\\
A_1(r)&=&-r^2A_3,\\
A_0(r)&=&r(4r^2c_0^2+2c_0+2rc_0'+4r^4c_0c_2+2r^2c_0f_1-r^6c_1^2-2r^2c_2-f_1+r^2f_2^2).
\end{eqnarray*}
Since above equality holds for any arbitrary $s$, one can conclude that it has a vanishing flag curvature if and only if the following equation system holds:
\begin{equation*}
\left\{ \begin{array}{l}
-(6rc_1+r^2c_1'+2r^3c_1f_1+2rf_1f_2-f_2')=0\\
\\
-(4rc_0^2+2c_0'+4r^3c_0c_2-r^5c_1^2-4rc_2-2r^3c_2f_1-rf_1^2+f_1'+rf_2^2)=0\\
\\
  4r^2c_0^2+2c_0+2rc_0'+4r^4c_0c_2+2r^2c_0f_1-r^6c_1^2-2r^2c_2-f_1+r^2f_2^2=0.
 \end{array} \right.
          \end{equation*}
Note that if multiplying the second equation by $r$, then adding the third equation yields
\[2c_0+2r^2c_0f_1+2r^2c_2+2r^4f_1c_2+r^2f_1^2-f_1-rf_1'=0.\]
Therefore the above equation system is equivalent to 
\begin{equation*}
\left\{ \begin{array}{l}
6rc_1+r^2c_1'+2r^3c_1f_1+2rf_1f_2-f_2'=0\\
\\
2c_0+2r^2c_0f_1+2r^2c_2+2r^4f_1c_2+r^2f_1^2-f_1-rf_1'=0\\
\\
  4r^2c_0^2+2c_0+2rc_0'+4r^4c_0c_2+2r^2c_0f_1-r^6c_1^2-2r^2c_2-f_1+r^2f_2^2=0.
 \end{array} \right.
          \end{equation*}

As the number of the equations is less than the number of the functions, there should exist infinite many solutions.  
Here for our purpose, we just pick up one solution 
\begin{equation*}
\left\{ \begin{array}{l}
f_1(r)=-\frac{1}{r^2}\\
f_2(r)=\frac{c}{r^2} \\
c_1(r)=-\frac{1}{r^4}\\
c_2(r)=-\frac{4r^4c_0^2+2r^3c_0'+c^2}{2r^4(2r^2c_0-1)}
 \end{array} \right.
          \end{equation*}
          depending on $c_0(r)$ and a constant $c$.
That is to say its spray coefficient $G^i=uPy^i+u^2Qx^i$ is given by
\[P=-\frac{s}{r^2}+\frac{c}{r^2}\sqrt{r^2-s^2} \] 
and \[ Q=c_0(r)-\frac{4r^4c_0^2+2r^3c_0'+c^2}{2r^4(2r^2c_0-1)}s^2-\frac{s}{r^4}\sqrt{r^2-s^2}.\]

Now we only need to solve the metric function $F=u\phi(r,s)$ from $P$ and $Q$. If let 
\[U:=\frac{s\phi+(r^2-s^2)\phi_s}{\phi},\qquad W:=\frac{s\phi_r+r\phi_s}{\phi},\]
then it is easy to see
\begin{equation*}
\phi_s=\frac{U-s}{r^2-s^2}\phi,\qquad\phi_r=\frac{1}{s}(W-\frac{r(U-s)}{r^2-s^2})\phi.
\end{equation*}
Therefore, $P$ and $Q$ can be expressed by $U$ and $W$ via
\begin{equation*}
\left\{ \begin{array}{l}
 P=-QU+\frac{W}{2r}\\
 \\
Q=\frac{1}{2rs}\frac{2rU-2rs-2r^2W+s(r^2-s^2)W_s+s^2W+sUW}{U^2-sU+(r^2-s^2)U_s}.   
 \end{array} \right.
          \end{equation*}
  Plugging $P$ and $Q$ into the above equation, one can solve
 \[U=\frac{2cr^2}{(c+1)s-\sqrt{r^2-s^2}(2r^2c_0-1)}\]
and
\[W=2r(P+UQ).\]
Hence
\begin{equation}\label{eq3}
\left\{ \begin{array}{ll}
(\ln\phi)_s=&\frac{(c+1)s^2-(2r^2c_0-1)s\sqrt{r^2-s^2}-2r^2c}{(c+1)s^3-(c+1)r^2s+(2c_0r^2-1)(r^2-s^2)\sqrt{r^2-s^2}}\\
\\
(\ln\phi)_r=&-\frac{\sqrt{r^2-s^2}(2c(c_0r^2-1)(c-1)(r^2-s^2)+r^2-s^2-4r^4c_0-8c_0^2r^4s^2+4c_0^2r^6+8c_0r^2s^2)}{r(\sqrt{r^2-s^2}(4c_0^2r^4-4c_0r^2+1)-2c_0s(cr^2s+r^2)+(c+1)s)(r^2-s^2)}\\
&-\frac{5cr^2s-2c^3r^2s-4cc_0'r^5s+12cc_0r^2s^3-14cc_0r^4s+4cc_0'r^3s^3+r^2s-4cs^3-2c_0r^4s+4c_0r^2s^3+2c^3s^3-2s^3}{r(\sqrt{r^2-s^2}(4c_0^2r^4-4c_0r^2+1)-2c_0s(cr^2s+r^2)+(c+1)s)(r^2-s^2)}.
\end{array}\right.
\end{equation}
It can be checked that $(\ln\phi)_{sr}=(\ln\phi)_{rs}$.
This means there exists $\phi(r,s)$ so that it is a solution of above equation system. Integrating the first equation of (\ref{eq3}) will obtain
\[\phi(r,s)=\exp\Big(\int_{0}^s\frac{(c+1)s^2-(2r^2c_0-1)s\sqrt{r^2-s^2}-2r^2c}{(c+1)s^3-(c+1)r^2s+(2c_0r^2-1)(r^2-s^2)\sqrt{r^2-s^2}}ds\Big)a(r).\]
In order to find out $a(r)$, notice that $\phi(r,0)=a(r)$. Thus $$\phi(r,0)_r=a'(r).$$
According to the second equation of (\ref{eq3}), we know $a(r)$ satisfies
\[\frac{a'}{a}=-\frac{4c_0r^2r^4-4r^2c_0+1+2c(2c_0r^2-1)(c-1)}{r(4c_0r^4-4r^2c_0+1)}.\]
Integrating above equality will yield
\[a(r)=\exp(\int-\frac{4c_0r^2r^4-4r^2c_0+1+2c(2c_0r^2-1)(c-1)}{r(4c_0r^4-4r^2c_0+1)}dr).\]
So the construction is completed.  To sum up, we have the following theorem.
\begin{theorem}If the Finsler surface is given by \[F=u \exp(\int_0^s\frac{(c+1)s^2-(2r^2c_0-1)s\sqrt{r^2-s^2}-2r^2c}{(c+1)s^3-(c+1)r^2s+(2c_0r^2-1)(r^2-s^2)\sqrt{r^2-s^2}} ds)a(r)\]
 where $r=\sqrt{x_1^2+x_2^2}$, $u=\sqrt{y_1^2+y_2^2}$, $s=x_1y_1+x_2y_2$ and $$a(r)=\exp(\int-\frac{4c_0r^2r^4-4r^2c_0+1+2c(2c_0r^2-1)(c-1)}{r(4c_0r^4-4r^2c_0+1)}dr),$$
$c_0$ is a smooth function of $r$ and $c\neq \frac{1}{3}$ is a constant, then it is non-Berwaldian, Landsbergian and has a vanishing flag curvature. 

\end{theorem} 
\begin{proof} It can be proved by a straightforward computation.
\end{proof}

\begin{remark}Actually one can check the Finsler surfaces in above theorem is not the Douglas metrics. It is not clear whether there is any Landsberg surface with a non-zero constant flag curvature.
\end{remark}

\begin{corollary} (R. Bryant) There exists a class of singular non-Berwaldian Finsler surfaces with J=0 and K=0. 
\end{corollary}

{\small DEPARTMENT OF MATHEMATICS, EAST CHINA NORMAL UNIVERSITY }

{\small SHANGHAI 200241, CHINA.}

{\small E-mail address: lfzhou@math.ecnu.edu.cn}

\end{document}